\documentclass[10pt]{amsart}
\usepackage{amsmath,amscd}
\usepackage{amsbsy}
\usepackage{amssymb}
\usepackage{amscd,amsthm}
\usepackage[all,cmtip]{xy}

\newtheorem{thm}{Theorem}
\numberwithin{thm}{section}
\newtheorem{lem}[thm]{Lemma}
\newtheorem{prop}[thm]{Proposition}

\newtheorem{rema}[thm]{Remark}

\newtheorem{defi}[thm]{Definition}

\newtheorem*{que2}{Question}

\begin{document}
\begin{center}
\huge{Rationality of inner twisted flags of type $A_n$}\\[1cm]
\end{center}

\begin{center}

\large{Sa$\mathrm{\check{s}}$a Novakovi$\mathrm{\acute{c}}$}\\[0,4cm]
{\small January 2020}\\[0,3cm]
\end{center}

\noindent{\small \textbf{Abstract}. 
We prove that $\mathrm{rcodim}(X)\geq 2 $ if $X$ is a rational inner twisted flag of type $A_n$.\\

\section{Introduction}
A potential measure for rationality was introduced by Bernardara and Bolognesi \cite{BBS} with the notion of categorical representability. We use the definition given in \cite{AB1S}. A $k$-linear triangulated category $\mathcal{T}$ is said to be \emph{representable in dimension $m$} if there is a semiorthogonal decomposition (see Section 3 for the definition) $\mathcal{T}=\langle \mathcal{A}_1,...,\mathcal{A}_n\rangle$ and for each $i=1,...,n$ there exists a smooth projective connected variety $Y_i$ with $\mathrm{dim}(Y_i)\leq m$, such that $\mathcal{A}_i$ is equivalent to an admissible subcategory of $D^b(Y_i)$. We set
\begin{eqnarray*}
	\mathrm{rdim}(\mathcal{T}):=\mathrm{min}\{m\mid \mathcal{T}\  \textnormal{is representable in dimension m}\},
\end{eqnarray*}
whenever such a finite $m$ exists. Let $X$ be a smooth projective $k$-variety. One says $X$ is \emph{representable in dimension} $m$ if $D^b(X)$ is representable in dimension $m$. We will use the following notations:
\begin{eqnarray*}
	\mathrm{rdim}(X):=\mathrm{rdim}(D^b(X)),\thickspace \mathrm{rcodim}(X):= \mathrm{dim}(X)-\mathrm{rdim}(X).
\end{eqnarray*}
Note that when the base field $k$ of a variety $X$ is not algebraically closed, the existence of $k$-rational points on $X$ and whether $X$ is $k$-rational is a major open question in arithmetic geometry. We recall the following question which was formulated in \cite{MBTS}:
\begin{que2}
	Let $X$ be a smooth projective variety over $k$ of dimension at least 2. Suppose $X$ is $k$-rational. Do we have $\mathrm{rcodim}(X)\geq 2$ ?
\end{que2} 
There are several results suggesting that this question has a positive answer (see \cite{AB1S}, \cite{ABS}, \cite{MBTS}, \cite{NO2Z}, \cite{NO3T}, \cite{NOOU} and references therein). In this context, we prove: 
\begin{thm}
Let $A$ be a central simple $k$-algebra of degree $n$ over a field of characteristic zero. We fix a sequence of integers $0<n_1<n_2<\cdots <n_r<n$ and let $X=\mathrm{BS}(n_1,...,n_r,A)$ be a twisted flag of dimension at least two. If $X$ is rational over $k$, then $\mathrm{rcodim}(X)\geq 2$.	
	\end{thm}

{\small \textbf{Acknowledgement}. I would like to thank Nikita Semenov for answering questions concerning twisted flag varieties. 
This research was conducted in the framework of the research training group GRK 2240: Algebro-geometric Methods in Algebra, Arithmetic and Topology, which is funded by the $\mathrm{DFG}$.}\\


\section{Inner twisted flags of type $A_n$}
Recall that a finite-dimensional $k$-algebra $A$ is called \emph{central simple} if it is an associative $k$-algebra that has no two-sided ideals other than $0$ and $A$ and if its center equals $k$. If the algebra $A$ is a division algebra it is called \emph{central division algebra}. Note that $A$ is a central simple $k$-algebra if and only if there is a finite field extension $k\subset L$, such that $A\otimes_k L \simeq M_n(L)$. This is also equivalent to $A\otimes_k \bar{k}\simeq M_n(\bar{k})$. An extension $k\subset L$ such that $A\otimes_k L\simeq M_n(L)$ is called splitting field for $A$. The \emph{degree} of a central simple algebra $A$ is defined to be $\mathrm{deg}(A):=\sqrt{\mathrm{dim}_k A}$. According to the \emph{Wedderburn Theorem}, for any central simple $k$-algebra $A$ there is an unique integer $n>0$ and a division $k$-algebra $D$ such that $A\simeq M_n(D)$. The division algebra $D$ is also central and unique up to isomorphism. The degree of the unique central division algebra $D$ is called the \emph{index} of $A$ and is denoted by $\mathrm{ind}(A)$. Two central simple algebras $A$ and $B$ are said to be \emph{Brauer-equivalent} if there are positive integers $r,s$ such that $M_r(A)\simeq M_s(B)$. 

A Brauer--Severi variety of dimension $n$ can also be defined as a scheme $X$ of finite type over $k$ such that $X\otimes_k L\simeq \mathbb{P}^n$ for a finite field extension $k\subset L$. A field extension $k\subset L$ for which $X\otimes_k L\simeq \mathbb{P}^n$ is called \emph{splitting field} of $X$. Clearly, $k^s$ and $\bar{k}$ are splitting fields for any Brauer--Severi variety. In fact, every Brauer--Severi variety always splits over a finite Galois extension. It follows from descent theory that $X$ is projective, integral and smooth over $k$. For details we refer to \cite{ARS} and \cite{GSS}.

To a central simple $k$-algebra $A$ one can also associate twisted forms of Grassmannians. Let $A$ be of degree $n$ and $1\leq d\leq n$. Consider the subset of $\mathrm{Grass}_k(d\cdot n, A)$ consisting of those subspaces of $A$ that are left ideals $I$ of dimension $d\cdot n$. This subset can be given the structure of a projective variety which turns out to be a generalized Brauer--Severi variety. It is denoted by $\mathrm{BS}(d,A)$. After base change to some splitting field $L$ of $A$ the variety $\mathrm{BS}(d,A)$ becomes isomorphic to $\mathrm{Grass}_L(d,n)$. If $d=1$ the generalized Brauer--Severi variety is the Brauer--Severi variety associated to $A$. For details see \cite{BLS}. 

We also recall the basics of generalized Brauer--Severi schemes (see \cite{LSW}). Let $X$ be a noetherian $k$-scheme and $\mathcal{A}$ a sheaf of Azumaya algebras of rank $n^2$ over $X$ (see \cite{GRO}, \cite{GRO1} for details on Azumaya algebras). For an integer $1\leq n_1<n$ the generalized Brauer--Severi scheme $p:\mathrm{BS}(n_1,\mathcal{A})\rightarrow X$ is defined as the scheme representing the functor $F:\mathrm{Sch}/X\rightarrow \mathrm{Sets}$, where $(\psi:Y\rightarrow X )$ is mapped to the set of left ideals $\mathcal{J}$ of $\psi^*\mathcal{A}$ such that $\psi^*\mathcal{A}/\mathcal{J}$ is locally free of rank $n(n-n_1)$. By definition, there is an \'etale covering $U\rightarrow X$ and a locally free sheaf $\mathcal{E}$ of rank $n$ with the following trivializing diagram:
\begin{displaymath}
\begin{xy}
\xymatrix{
	\mathrm{Grass}(n_1,\mathcal{E}) \ar[r]^{\pi} \ar[d]_{q}    &   \mathrm{BS}(n_1,\mathcal{A}) \ar[d]^{p}                   \\
	U \ar[r]^{g}             &   X             
}
\end{xy}
\end{displaymath}
In the same way one defines the twisted relative flag $\mathrm{BS}(n_1,...,n_r,\mathcal{A})$ as the scheme representing the functor $F:\mathrm{Sch}/X\rightarrow \mathrm{Sets}$, where $(\psi:Y\rightarrow X )$ is mapped to the set of left ideals $\mathcal{J}_1\subset...\subset \mathcal{J}_r$ of $\psi^*\mathcal{A}$ such that $\psi^*\mathcal{A}/\mathcal{J}_i$ is locally free of rank $n(n-n_i)$. As for the generalized Brauer--Severi schemes, there is an \'etale covering $U\rightarrow X$ and a locally free sheaf $\mathcal{E}$ of rank $n$ with diagram
\begin{displaymath}
\begin{xy}
\xymatrix{
	\mathrm{Flag}_U(n_1,...,n_r,\mathcal{E}) \ar[r]^{\pi} \ar[d]_{q}    &   \mathrm{BS}(n_1,...,n_r,\mathcal{A}) \ar[d]^{p}                   \\
	U \ar[r]^{g}             &   X             
}
\end{xy}
\end{displaymath}
Note that the usual Brauer--Severi schemes are obtained from the generalized one by setting $n_1=1$. In this case one has a well known one-to-one correspondence between sheaves of Azumaya algebras of rank $n^2$ on $X$ and Brauer--Severi schemes of relative dimension $n-1$ via $\check{H}^1(X_{et}, \mathrm{PGL}_n)$ (see \cite{GRO}). Note that if the base scheme $X$ is a point a sheaf of Azumaya algebras on $X$ is a central simple $k$-algebra and the generalized Brauer--Severi schemes are the generalized Brauer--Severi varieties from above. Consider a twisted flag $X=SB(n_1,...,n_r, A) \rightarrow \mathrm{Spec}(k)$. Such an $X$ is an \emph{inner form} of a partial flag variety.
That is, there is a cartesian square of the form
\begin{displaymath}
\begin{xy}
\xymatrix{
	\mathrm{Grass}_L(n_1,...,n_r,V) \ar[r]^{\pi} \ar[d]_{q}    &   \mathrm{BS}(n_1,...,n_r, A) \ar[d]^{p}                   \\
	\mathrm{Spec}(L) \ar[r]^{\pi}             &   \mathrm{Spec}(k)             
}
\end{xy}
\end{displaymath}
where $L/k$ is a Galois extension and the 1-cocycle
$\mathrm{Gal}(L/k)\rightarrow \mathrm{Aut}(\mathrm{Grass}_L(n_1,...,n_r,V))$ factors through $\mathrm{PGL}(V)$.

\section{Semiorthogonal decompositions} 
Let $\mathcal{D}$ be a triangulated category and $\mathcal{C}$ a triangulated subcategory. The subcategory $\mathcal{C}$ is called \emph{thick} if it is closed under isomorphisms and direct summands. For a subset $A$ of objects of $\mathcal{D}$ we denote by $\langle A\rangle$ the smallest full thick subcategory of $\mathcal{D}$ containing the elements of $A$. 
For a smooth projective variety $X$ over $k$, we denote by $D^b(X)$ the bounded derived category of coherent sheaves on $X$. Moreover, if $B$ is an associated $k$-algebra, we write $D^b(B)$ for the bounded derived category of finitely generated left $B$-modules.

\begin{defi}
	\textnormal{Let $A$ be a division algebra over $k$, not necessarily central. An object $\mathcal{E}^{\bullet}\in D^b(X)$ is called \emph{$A$-exceptional} if $\mathrm{End}(\mathcal{E}^{\bullet})=A$ and $\mathrm{Hom}(\mathcal{E}^{\bullet},\mathcal{E}^{\bullet}[r])=0$ for $r\neq 0$. By \emph{generalized exceptional object}, we mean $A$-exceptional for some division algebra $A$ over $k$. 
		If $A=k$, the object $\mathcal{E}^{\bullet}$ is called \emph{exceptional}.} 
\end{defi}
\begin{defi}
	\textnormal{A totally ordered set $\{\mathcal{E}^{\bullet}_1,...,\mathcal{E}^{\bullet}_n\}$ of generalized exceptional objects on $X$ is called an \emph{generalized exceptional collection} if $\mathrm{Hom}(\mathcal{E}^{\bullet}_i,\mathcal{E}^{\bullet}_j[r])=0$ for all integers $r$ whenever $i>j$. A generalized exceptional collection is \emph{full} if $\langle\{\mathcal{E}^{\bullet}_1,...,\mathcal{E}^{\bullet}_n\}\rangle=D^b(X)$ and \emph{strong} if $\mathrm{Hom}(\mathcal{E}^{\bullet}_i,\mathcal{E}^{\bullet}_j[r])=0$ whenever $r\neq 0$. If the set $\{\mathcal{E}^{\bullet}_1,...,\mathcal{E}^{\bullet}_n\}$ consists of exceptional objects it is called \emph{exceptional collection}.}
\end{defi}

 Recall that a full thick triangulated subcategory $\mathcal{D}$ of $D^b(X)$ is called \emph{admissible} if the inclusion $\mathcal{D}\hookrightarrow D^b(X)$ has a left and right adjoint functor. 
\begin{defi}
	\textnormal{Let $X$ be a smooth projective variety over $k$. A sequence $\mathcal{D}_1,...,\mathcal{D}_n$ of full admissible triangulated subcategories of $D^b(X)$ is called \emph{semiorthogonal} if $\mathcal{D}_j\subset \mathcal{D}_i^{\perp}=\{\mathcal{F}^{\bullet}\in D^b(X)\mid \mathrm{Hom}(\mathcal{G}^{\bullet},\mathcal{F}^{\bullet})=0$, $\forall$ $ \mathcal{G}^{\bullet}\in\mathcal{D}_i\}$ for $i>j$. Such a sequence defines a \emph{semiorthogonal decomposition} of $D^b(X)$ if the smallest thick full subcategory containing all $\mathcal{D}_i$ equals $D^b(X)$.}
\end{defi}
\noindent
For a semiorthogonal decomposition we write $D^b(X)=\langle \mathcal{D}_1,...,\mathcal{D}_n\rangle$.

\begin{rema}
	\textnormal{Let $\mathcal{E}^{\bullet}_1,...,\mathcal{E}^{\bullet}_n$ be a full generalized exceptional collection on $X$. It is easy to verify that by setting $\mathcal{D}_i=\langle\mathcal{E}^{\bullet}_i\rangle$ one gets a semiorthogonal decomposition $D^b(X)=\langle \mathcal{D}_1,...,\mathcal{D}_n\rangle$.}
\end{rema}

\section{Proof of Theorem 1.1}
\begin{lem}
	Let $X$ be the inner twisted flag from Theorem 1.1. Then $D^b(X)$ admits a semiorthogonal decomposition with components being equivalent to $D^b(A^{\otimes j})$ for suitable positive integers $j\geq 0$.
\end{lem}
\begin{proof}
	On a relative flag $\mathrm{Grass}(n_1,...,n_r,\mathcal{E})$ one has tautological subbundles sitting in the following sequence
	\begin{eqnarray*}
	0\longrightarrow \mathcal{R}_1\longrightarrow \cdots \longrightarrow \mathcal{R}_r\longrightarrow q^*\mathcal{E}\longrightarrow \mathcal{T}_1\longrightarrow \cdots \longrightarrow \mathcal{T}_r\longrightarrow 0.	
		\end{eqnarray*}
	In our case $\mathcal{E}$ is a $k$-vector space of dimension $n$ denoted by $V$. Recall that $X=\mathrm{BS}(n_1,...,n_r,A)$ is the inner twisted flag obtained from $\mathrm{Grass}(n_1,...,n_r,V)$ by descent.
Denote by $\alpha(1),...,\alpha(r)$ partitions of the forms
\begin{eqnarray*}
	(\alpha_1,...,\alpha_{l_1}), (\alpha_1,...,\alpha_{l_2}),...,(\alpha_1,...,\alpha_{l_r}),
	\end{eqnarray*}
with $0\leq \alpha_i\leq l_2-l_1,...,0\leq \alpha_i\leq l_m-l_{m-1}, 0\leq \alpha_i\leq n-l_m$. Let $|\alpha|=\sum_{i=1}^r|\alpha(i)|$, where $|\alpha(j)|=\sum_{s=1}^{l_{j}}\alpha_s$ for $1\leq j\leq r$. Then the sheaf $(S^{\alpha(1)}\mathcal{R}_1\otimes\cdots \otimes S^{\alpha(r)}\mathcal{R}_r)^{\oplus n\cdot |\alpha|}$ descents to a sheaf $\mathcal{J}_{\alpha}$ on $X$ (see \cite{LSW}, Section 4). 
As in the case of generalized Brauer--Severi varieties, one can show that $\mathrm{End}(\mathcal{J}_{\alpha})$ is Brauer-equivalent to $A^{\otimes|\alpha|}$.
Note that $D^b(\mathrm{End}(\mathcal{J}_{\alpha}))$ is equivalent to $D^b(A^{\otimes |\alpha|})$ by Morita equivalence. Now it is proved in \cite{BAEK} that the subcategories $\langle \mathcal{J}_{\alpha}\rangle \simeq D^b(\mathrm{End}(\mathcal{J}_{\alpha}))$ form a semiorthogonal decomposition. 
\end{proof} 
\begin{prop}
	Let $X$ be the inner twisted flag from above. Then $\mathrm{rdim}(X)\leq \mathrm{ind}(A)-1$.
\end{prop}
\begin{proof}
According to Lemma 4.1, $D^b(X)$ admits a semiorthogonal decomposition with components being equivalent to $D^b(A^{\otimes j})$ for suitable positive integers $j\geq 0$. Now let $A=M_s(D)$ according to the Wedderburn theorem. Since any of the derived categories $D^b(A^{\otimes j})$ can be embedded into $D^b(Y)$, where $Y$ is the Brauer--Severi variety corresponding to $D$ (see \cite{BERS}), the assertion follows from the fact that $\mathrm{deg}(D)=\mathrm{ind}(A)$ and $\mathrm{dim}(Y)=\mathrm{ind}(A)-1$.	
\end{proof}
\begin{lem}
	Let $\alpha\geq1$ be an integer. Then $\alpha x^2-x-1\geq 0$ for all $x\geq 2$.
\end{lem}
\begin{proof}
	The function $\alpha x^2-x-1$ is strictly monotonically increasing for $x> 1/2\alpha$ and has its positive root at $(1/2\alpha)\cdot (1+\sqrt{1+4\alpha})$ which is $<2$ for all $\alpha \geq 1$.
\end{proof}
\begin{proof}(of Theorem 1.1)\\
Recall the well known fact that $X(k)\neq \emptyset$ iff $\mathrm{ind}(A)\mid d$, where $d=\mathrm{gcd}(n,n_1,...,n_r)$ denotes the greatest common divisor. So if $X$ is $k$-rational, it has a $k$-rational point and therefore $\mathrm{ind}(A)\mid d$. Now let $n=b\cdot \mathrm{ind}(A)$ and $n_i=b_i\cdot \mathrm{ind}(A)$ for $i=1,...,r$. Then  $b,b_1,...,b_r$ are positive integers satisfying $0<b_1<b_2<\cdots <b_r<b$. This implies that $a\geq 1$ for 
\begin{eqnarray*}
a=b_1(b-b_1)+\sum_{i=2}^r(b_i-b_{i-1})(b-b_i).
\end{eqnarray*}  
Assume $\mathrm{ind}(A)>1$ and let $x=\mathrm{ind}(A)$. With $\alpha=a$, Lemma 4.3 implies
\begin{eqnarray*}
a\cdot \mathrm{ind}(A)^2-\mathrm{ind}(A)-1\geq 0.
\end{eqnarray*} 
It is easy to see that 
\begin{eqnarray*}
a\cdot \mathrm{ind}(A)^2=n_1(n-n_1)+\sum_{i=2}^r(n_i-n_{i-1})(n-n_i).
\end{eqnarray*}
This implies 
\begin{eqnarray*}
n_1(n-n_1)+\sum_{i=2}^r(n_i-n_{i-1})(n-n_i)\geq (\mathrm{ind}(A)-1)+2.
\end{eqnarray*} 
Now, since 
\begin{eqnarray*}
\mathrm{dim}(X)=n_1(n-n_1)+\sum_{i=2}^r(n_i-n_{i-1})(n-n_i)
\end{eqnarray*}
and since $\mathrm{rdim}(X)\leq \mathrm{ind}(A)-1$ according to Proposition 4.2, we conclude
\begin{eqnarray*}
	\mathrm{rcodim}(X)=\mathrm{dim}(X)-\mathrm{rdim}(X)\geq 2.
	\end{eqnarray*}
It remains to consider the case $\mathrm{ind}(A)=1$. In this case the inner twisted flag $X$ is split, i.e. is isomorphic to $\mathrm{Flag}_k(n_1,...,n_r,V)$, where $V$ is a $n$-dimensional $k$-vector space. From Kapranov \cite{KA2S}, we know that $X$ admits a full exceptional collection. But then $\mathrm{rdim}(X)=0$ according to \cite{AB1S}, Proposition 6.1.6 and hence $\mathrm{rcodim}(X)\geq 2$.
\end{proof}
\begin{rema}
	\textnormal{If the inner twisted flag $X$ from Theorem 1.1 is of dimension one, it is actually a Brauer--Severi curve. In this case \cite{NO2Z} shows that rationality of $X$ is equivalent to $\mathrm{rdim}(X)=0$. Hence $\mathrm{rcodim}(X)=1$. This is one reason why the question in the introduction assumes $\mathrm{dim}(X)\geq 2$. }
\end{rema}

{\small MATHEMATISCHES INSTITUT, HEINRICH--HEINE--UNIVERSIT\"AT 40225 D\"USSELDORF, GERMANY}\\
E-mail adress: novakovic@math.uni-duesseldorf.de

\end{document}